\documentclass[preprint,12pt]{elsarticle}

\usepackage{amsmath,amssymb,amsthm}
\IfFileExists{lmodern.sty}{\usepackage{lmodern}\usepackage{microtype}}{\usepackage[expansion=false]{microtype}}
\usepackage{booktabs}
\usepackage{xurl}
\usepackage[hidelinks]{hyperref}

\newtheorem{maintheorem}{Theorem}

\newtheorem{lemma}{Lemma}[section]
\newtheorem{proposition}[lemma]{Proposition}
\newtheorem{corollary}[lemma]{Corollary}
\theoremstyle{remark}
\newtheorem{remark}[lemma]{Remark}
\theoremstyle{plain}

\newcommand{\doi}[1]{\href{https://doi.org/#1}{doi:\nolinkurl{#1}}}
\newcommand{\Q}{\mathbb Q}
\newcommand{\R}{\mathbb R}
\newcommand{\SJ}{SJ}
\DeclareMathOperator{\rank}{rank}
\DeclareMathOperator{\im}{im}
\DeclareMathOperator{\Row}{Row}
\DeclareMathOperator{\tr}{tr}
\DeclareMathOperator{\End}{End}
\DeclareMathOperator{\SYT}{SYT}

\begin{document}
\begin{frontmatter}

\title{A polynomial criterion for Jordan elements in a free associative algebra}

\author[1,2]{F.~Mashurov\corref{cor1}}
\ead{f.mashurov@gmail.com}
\cortext[cor1]{Corresponding author}

\author[3]{B.~Sartayev}
\ead{baurjai@gmail.com}

\address[1]{SICM, Southern University of Science
and Technology, Shenzhen, 518055, China}

\address[2]{SDU University, Kaskelen, Kazakhstan}

\address[3]{Narxoz University, Almaty, Kazakhstan}

\begin{abstract}
Let $J$ be the Jordan subalgebra of the free associative algebra
$\Q\langle X\rangle$ generated by $X$ and $1$ under
$a\cdot b=\tfrac12(ab+ba)$. We give a polynomial criterion for an
associative polynomial to belong to $J$. For every $n\ge1$, Robbins'
recursion gives a family of operators $T_\alpha\in\Q[S_n]$, indexed by
compositions of $n$ into parts $1$ and $2$. We prove that right
multiplication by $U_n=\sum_\alpha T_\alpha^\dagger T_\alpha$ has image
$\SJ(n)$, the space of multilinear Jordan elements. If its characteristic
polynomial is $t^{e_n}q_n(t)$, where $q_n(0)\ne0$, then
$\Pi_n=I-q_n(U_n)/q_n(0)$ is an idempotent and
$f\in\SJ(n)$ if and only if $fq_n(U_n)=0$, or equivalently $f\Pi_n=f$.
The same operators give the criterion for arbitrary homogeneous
polynomials of degree $n$. We also prove
$\SJ(n)\cong\bigoplus_{\lambda\vdash n}r_\lambda S^\lambda$, where
$r_\lambda=\rank\rho_\lambda(U_n)$. Thus
$\dim\SJ(n)=\sum_\lambda d_\lambda r_\lambda$ and
$e_n=\sum_\lambda d_\lambda(d_\lambda-r_\lambda)$, with
$d_\lambda=\dim S^\lambda$. We give a recursion for the corresponding
row spaces and use it to compute the multiplicities over $\Q$ through
degree $12$. Exact block characteristic polynomials through degree $7$
and direct computations in small degrees illustrate the criterion.
\end{abstract}

\end{frontmatter}

\section{Introduction}
\label{sec:introduction}

Let $A=\Q\langle X\rangle$ be the free associative algebra on a set $X$.
Two classical subalgebras of $A$ are generated by $X$ under the commutator
$[a,b]=ab-ba$ and under the Jordan product
\[
 a\cdot b=\tfrac12(ab+ba).
\]
The first is the free Lie algebra on $X$. For it, the recognition problem
(given $f\in A$, decide whether $f$ is a Lie element) has two classical
solutions in characteristic zero: Friedrichs' criterion, which tests whether
$f$ is primitive for the coproduct $x\mapsto x\otimes1+1\otimes x$, and the
Dynkin--Specht--Wever criterion, which, for a homogeneous polynomial $f$ of degree $n\ge1$, tests
whether $f\theta=nf$ for the operator $\theta$ replacing every word $x_1\cdots x_n$ by the left-normed
commutator $[\ldots[x_1,x_2],\ldots,x_n]$
\cite{Friedrichs,Dynkin,Specht,Wever}; see Reutenauer \cite{Reutenauer}.
The recognition problem is accompanied by a complete description of the
multilinear component $\mathrm{Lie}(n)$ as a module over the symmetric
group $S_n$, due to Klyachko and Kra\'skiewicz--Weyman
\cite{Klyachko,KraskiewiczWeyman}. In this paper, we consider the corresponding two questions for the
Jordan subalgebra: a polynomial criterion for Jordan elements and a
method for determining its multilinear components as $S_n$-modules.

Jordan algebras were introduced in the algebraic formulation of quantum
mechanics by Jordan, von Neumann and Wigner, with Albert's contribution
appearing in the same volume \cite{JordanNeumannWigner,Albert1934}. The Jordan algebras that embed into associative algebras with the
symmetrized product are called \emph{special}, and the question of which Jordan
algebras are special, or satisfy the identities of special ones, runs
through the whole subject. Shirshov proved that the free Jordan algebra on two generators is
special \cite{Shirshov}, and Cohn showed that every homomorphic image
of the free special Jordan algebra on two generators is special
\cite{Cohn}. Glennie
found identities of degree $8$ that hold in every special Jordan algebra
but not in the free Jordan algebra on three generators \cite{Glennie}; see
Sverchkov \cite{Sverchkov} for the identities in three variables. Zelmanov's structure theory of prime nondegenerate Jordan
algebras distinguishes the special cases of Hermitian and Clifford
type and the exceptional Albert case
\cite{Zelmanov1979,Zelmanov1983}; the quadratic version is due to
McCrimmon and Zelmanov \cite{McCrimmonZelmanov}, and the monographs
\cite{Jacobson,McCrimmon,Zhevlakov} present the surrounding theory. In
all of these developments the basic object is the Jordan subalgebra of
an associative algebra generated by a set of elements, and the most
elementary question about it is to recognize its elements. Here we study
this question for the free associative algebra.

\subsection*{The problem}

Let $J$ be the Jordan subalgebra of $A^{(+)}=(A,\cdot)$ generated by $X$
and $1$; this is the free special Jordan algebra on $X$. For distinct
variables $z_1,\ldots,z_n\in X$ put
\[
 V_n=\operatorname{span}_{\Q}\{z_{\sigma(1)}\cdots z_{\sigma(n)}:\sigma\in S_n\},
 \qquad \SJ(n)=J\cap V_n .
\]
Thus $\SJ(n)$ is the space of multilinear Jordan elements of degree $n$;
it is stable under the action of $S_n$ by relabelling the variables.
The recognition problem is to decide whether an associative polynomial
lies in the span of the expansions of Jordan monomials. We seek a
criterion expressed by an explicit operator on associative words.

The first systematic study is due to Robbins \cite{Robbins}. He singled
out the \emph{simple} Jordan elements, the span of the left-normed
products $$x_1\cdot x_2\cdot\ldots\cdot x_n=(\cdots((x_1\cdot x_2)\cdot
x_3)\cdots)\cdot x_n$$ of generators, and proved a Specht--Wever type
criterion for them: if $T$ denotes the linear operator on $A$ with
$1T=1$ and $(x_1\cdots x_n)T=x_1\cdot x_2\cdot\ldots\cdot x_n$ for
$x_i\in X$, then $T$ is diagonalizable on each homogeneous component
$A_n$, all its nonzero eigenvalues belong to the set of numbers $4^{-j}$ for
$0\le j\le\lfloor (n-1)/2\rfloor$, and $a\in A_n$ is a simple Jordan
element if and only if \cite[Theorem~8]{Robbins}
\[
 a\prod_{j=0}^{\lfloor(n-1)/2\rfloor}(T-4^{-j}I)=0.
\]
Simple Jordan elements form a proper
subspace of $\SJ(n)$ already in degree $4$, where the dimensions are $9$ and
$11$: for instance $(z_1\cdot z_2)\cdot(z_3\cdot z_4)$ is not a simple
Jordan element. The full multilinear Jordan space is reached by
Robbins' recursion for the homogeneous components $J_m$ of $J$,
\begin{equation}
 J_{m+2}=J_{m+1}\cdot J_1+J_m\cdot J_2\qquad(m\ge0),
 \label{eq:Robbins-intro}
\end{equation}
which says that factors of degree two, but no higher, must be admitted
in left-normed products. Equation \eqref{eq:Robbins-intro} is the
starting point of this paper.

The full multilinear Jordan spaces have also been approached from the
side of the free (nonspecial) Jordan algebra. Kashuba and Mathieu
\cite{KashubaMathieu} conjectured a formula for the $S_n$-module
structure of the multilinear components $\mathrm{Jord}(n)$ of the free
Jordan algebra, and Dotsenko and Hentzel \cite{DotsenkoHentzel} computed
these modules for $n\le10$ and showed that its dimension prediction fails in degree $19$ for the
free Jordan algebra on two generators. Since $\SJ(n)$ is the image of $\mathrm{Jord}(n)$ under the
canonical map onto the free special Jordan algebra, the modules
$\mathrm{Jord}(n)$ bound the modules $\SJ(n)$ from above; we compare the
two in Section~\ref{sec:examples}.

\subsection*{Results}

Our construction uses Robbins' recursion to pass from Jordan products
to a family of explicit operators whose images span $\SJ(n)$. Adjoint
operators then convert the orthogonal complement of this sum into an
intersection of kernels. A sum of squared norms detects that
intersection and produces a single self-adjoint operator with image
$\SJ(n)$. A polynomial in this operator separates its kernel from its
image and gives the desired projection. For multiplicities, the
branching rule carries the same recursion directly into the row
spaces of the irreducible blocks.

Let $C_n$ be the set of compositions $\alpha=(\alpha_1,\ldots,\alpha_k)$ of
$n$ with all parts equal to $1$ or $2$. For $\alpha\in C_n$ define a
linear operator $T_\alpha$ on $A_n$: divide a word of length $n$ into
consecutive blocks of lengths $\alpha_1,\ldots,\alpha_k$, replace each
two-letter block $xy$ by $x\cdot y$, and take the left-normed Jordan
product of the resulting factors. Expanding the Jordan products shows
that $T_\alpha$ permutes positions with rational coefficients, so
$T_\alpha$ is an element of the group algebra $\Q[S_n]$ acting on words
from the right (Section~\ref{sec:preliminaries}). The recursion
\eqref{eq:Robbins-intro} implies $\SJ(n)=\sum_\alpha V_nT_\alpha$. For
$T=\sum_\sigma c_\sigma\sigma\in\Q[S_n]$ put
\[
 T^\dagger=\sum_\sigma c_\sigma\sigma^{-1},
 \qquad
 U_n=\sum_{\alpha\in C_n}T_\alpha^\dagger T_\alpha .
\]
The involution $\dagger$ gives the adjoint for the inner product on $V_n$ in which
the multilinear words form an orthonormal basis, and $U_n$ is the Gram
operator of the family $(T_\alpha)$. Writing $R_{U_n}$ for the operator
$f\mapsto fU_n$ on $V_n$, let
\[
 \chi_n(t)=\det(tI-R_{U_n})=t^{e_n}q_n(t),\qquad q_n(0)\ne0,
\]
so that $e_n$ is the order of vanishing of $\chi_n$ at $0$ and $q_n$ is
the monic polynomial whose roots are the nonzero eigenvalues of
$R_{U_n}$, each with its multiplicity in $\chi_n$.

\begin{maintheorem}[Polynomial criterion]
\label{thm:A}
For every $n\ge1$ the operator $R_{U_n}$ is self-adjoint, and it is
positive semidefinite after extension of scalars to $\R$. Its image is
the multilinear Jordan space and its nullity is the exponent $e_n$:
\[
 \SJ(n)=V_nU_n,\qquad e_n=\dim\ker R_{U_n}.
\]
The element
\[
 \Pi_n=I-\frac{q_n(U_n)}{q_n(0)}\in\Q[S_n]
\]
satisfies $\Pi_n^2=\Pi_n$ and $V_n\Pi_n=\SJ(n)$, and for every $f\in V_n$
\[
 f\in\SJ(n)
 \iff fq_n(U_n)=0
 \iff f\Pi_n=f .
\]
\end{maintheorem}

The second result computes the image through the irreducible
representations of $S_n$. Under the correspondence
\[
 z_{\sigma(1)}\cdots z_{\sigma(n)}\longleftrightarrow\sigma,
 \qquad V_n\simeq\Q[S_n],
\]
with permutations composed by $(\sigma\tau)(i)=\sigma(\tau(i))$, relabelling
the variables is left multiplication and permuting positions is right
multiplication. \emph{Throughout the paper, all group-algebra operators act
by right multiplication, while $V_n$ is regarded as the regular left
$S_n$-module.} Since the two actions commute, $R_{U_n}$ is a homomorphism
of left $S_n$-modules. For a partition $\lambda\vdash n$ let $S^\lambda$ be
the rational Specht module, let $\rho_\lambda\colon\Q[S_n]\to\End_\Q(S^\lambda)$
be a matrix representation on it, and put
\[
 d_\lambda=\dim S^\lambda,\qquad r_\lambda=\rank\rho_\lambda(U_n).
\]

\begin{maintheorem}[Irreducible reduction]
\label{thm:B}
For every $n\ge1$ there is an isomorphism of left $S_n$-modules
\[
 \SJ(n)\cong\bigoplus_{\lambda\vdash n}r_\lambda S^\lambda .
\]
Consequently
\[
 \dim\SJ(n)=\sum_{\lambda\vdash n}d_\lambda r_\lambda,
 \qquad
 e_n=\sum_{\lambda\vdash n}d_\lambda(d_\lambda-r_\lambda),
\]
and $\chi_n(t)=\prod_{\lambda\vdash n}\det\bigl(tI-\rho_\lambda(U_n)\bigr)^{d_\lambda}$.
\end{maintheorem}

Theorem~\ref{thm:B} follows by applying the Wedderburn decomposition to
right multiplication by $U_n$. In each summand
$S^\lambda\otimes(S^\lambda)^*$, this multiplication acts on the second
factor. Its rank is the multiplicity of $S^\lambda$ in $\SJ(n)$.

The irreducible blocks are much smaller than the regular matrix. In
degree $9$, the regular matrix has order $362\,880$ and the largest
block has order $216$; in degree $10$, the corresponding orders are
$3\,628\,800$ and $768$. The block characteristic polynomials determine
$\chi_n$, and hence $q_n$ and $\Pi_n$. If only the multiplicities are
needed, we use the row-space recursion in
Proposition~\ref{prop:row-recursion-main}, which does not require forming
$U_n$. In degree $12$, the largest block has order $7700$, compared
with $12!=479\,001\,600$ for the regular matrix. We compute all
multiplicities through degree $12$ over $\Q$ and the block
characteristic polynomials through degree $7$.

A finite membership algorithm is already obtained by expanding all
Jordan monomials of a fixed degree and solving a linear system. Our
criterion expresses membership by a polynomial in one explicitly
defined element of $\Q[S_n]$. Here \emph{polynomial} refers to the
expression $q_n(U_n)$, not to a bound on the running time. Likewise, the
rank formula gives an effective procedure for the multiplicities,
rather than a closed combinatorial formula for them. The same $U_n$,
$q_n$ and $\Pi_n$ apply to homogeneous polynomials with repeated
variables; see Corollary~\ref{cor:homogeneous}.

\subsection*{Organization}

Section~\ref{sec:preliminaries} fixes notation and the action conventions.
Section~\ref{sec:recursion} recalls Robbins' recursion and establishes the spanning property of the $T_\alpha$. Theorem~\ref{thm:A} is proved in
Section~\ref{sec:criterion} and Theorem~\ref{thm:B} in
Section~\ref{sec:representation}, which ends with the recursion used to
compute the multiplicities. Section~\ref{sec:examples} works out degrees
three and four in full, lists the computed dimensions, and compares them
with the free Jordan algebra.  \ref{app:q5} records the factorization of $q_5$, and
 \ref{app:computations} specifies the computations and the
files available in the GitHub repository~\cite{JordanCode}.
 \ref{app:multiplicities} lists the decompositions in degrees
$8$ through $12$.

\section*{Acknowledgments}
We thank Askar Dzhumadil'daev and Nurlan Ismailov, who posed this open problem
at the Algebra Seminar of Al-Farabi Kazakh National University (2014-2017).
 The first author thanks Efim Zelmanov and Iryna Kashuba for their courses
at SUSTech, which provided further motivation for this work. The first author also expresses his gratitude to Professors Efim Zelmanov, Ivan Shestakov, Olivier Mathieu and Vladimir Dotsenko  for useful discussions.

\section{Preliminaries and conventions}
\label{sec:preliminaries}

We consider all vector spaces and algebras over $\Q$. Let
$A=\Q\langle X\rangle$ be the free unital associative algebra on $X$.
Its basis consists of the words in $X$, including the empty word $1$.
Write
\[
 A=\bigoplus_{m\ge0}A_m,\qquad A_0=\Q1,
\]
where $A_m$ is the span of the words of length $m$. For the multilinear
components, we assume that $X$ contains distinct generators
$z_1,z_2,\ldots$. The homogeneous criterion will also apply when $X$ is
finite.

On $A$ we define the Jordan product
\begin{equation}
 a\cdot b=\tfrac12(ab+ba).
 \label{eq:Jordan-product}
\end{equation}
The resulting algebra $A^{(+)}$ is commutative and satisfies
$(a^2\cdot b)\cdot a=a^2\cdot(b\cdot a)$. Let $J$ be its subalgebra
generated by $X$ and $1$, and put $J_m=J\cap A_m$. We call elements of
$J$ Jordan elements. For subspaces $B,C\subseteq A$, the notation
$B\cdot C$ means the span of all $b\cdot c$ with $b\in B$ and $c\in C$.

A Jordan monomial is a parenthesized Jordan product of generators, with
repetitions allowed. We write products without parentheses in the
left-normed sense:
\[
 a_1\cdot a_2\cdot\ldots\cdot a_k
 =(\cdots((a_1\cdot a_2)\cdot a_3)\cdots)\cdot a_k.
\]
A one-factor product is the factor itself. We compose right operators
in the order in which they are applied: $bST=(bS)T$.

\subsection*{The two symmetric-group actions}

We compose permutations by
\begin{equation}
 (\sigma\tau)(i)=\sigma(\tau(i)).
 \label{eq:permutation-composition}
\end{equation}
A permutation of positions acts on the right on $A_n$:
\[
 (x_1\cdots x_n)\tau=x_{\tau(1)}\cdots x_{\tau(n)}.
\]
Variable relabelling acts on the left on $V_n$:
\[
 g\cdot(z_{i_1}\cdots z_{i_n})=z_{g(i_1)}\cdots z_{g(i_n)}.
\]
If $w_\sigma=z_{\sigma(1)}\cdots z_{\sigma(n)}$, then
\[
 g\cdot w_\sigma=w_{g\sigma},\qquad
 w_\sigma\tau=w_{\sigma\tau}.
\]
Thus $w_\sigma\leftrightarrow\sigma$ identifies $V_n$ with the regular
$\Q[S_n]$-bimodule; see \cite{Sagan}. In particular, the two actions
commute.

For $P\in\Q[S_n]$, we write $R_P(f)=fP$ for the corresponding right
operator. With the convention above,
$R_{PQ}=R_PR_Q$. Each $R_P$ is a homomorphism of the regular left
$S_n$-module. This right action is faithful: if $fP=0$ for all $f\in V_n$,
then $P=0$ by taking the basis word corresponding to the identity
permutation. Since relabelling preserves Jordan products, $\SJ(n)$ is a
left $S_n$-submodule of $V_n$.

\subsection*{The adjoint}

We make the word basis of $V_n$ orthonormal:
\[
 \left\langle\sum_\sigma a_\sigma w_\sigma,
 \sum_\sigma b_\sigma w_\sigma\right\rangle
 =\sum_\sigma a_\sigma b_\sigma.
\]
This form is positive definite over $\Q$ and remains so over $\R$.
For $P=\sum_\sigma c_\sigma\sigma\in\Q[S_n]$, put
$P^\dagger=\sum_\sigma c_\sigma\sigma^{-1}$.

A permutation of positions permutes the orthonormal word basis, so its
adjoint is its inverse. We use the standard finite-dimensional
identities $(\im B)^\perp=\ker B^*$ and
$\im B=(\ker B^*)^\perp$; see
\cite[Sections~6C and~7A]{Axler2024}. The following lemma records their
form for our right action.

\begin{lemma}
\label{lem:group-algebra-adjoint}
For $P,Q\in\Q[S_n]$ and $f,h\in V_n$, we have
\[
 (P^\dagger)^\dagger=P,\qquad
 (PQ)^\dagger=Q^\dagger P^\dagger,\qquad
 \langle fP,h\rangle=\langle f,hP^\dagger\rangle.
\]
Consequently,
\begin{equation}
 (V_nP)^\perp=\ker R_{P^\dagger}.
 \label{eq:image-adjoint-kernel}
\end{equation}
\end{lemma}
\begin{proof}
The first two identities are the inverse rules in the group algebra.
For basis words,
\[
 \langle w_\sigma\tau,w_\eta\rangle
 =\delta_{\sigma\tau,\eta}
 =\delta_{\sigma,\eta\tau^{-1}}
 =\langle w_\sigma,w_\eta\tau^{-1}\rangle.
\]
Linearity gives the adjoint identity. Since the scalar product is
nondegenerate, a vector $h$ is orthogonal to every $fP$ exactly when
$hP^\dagger=0$, proving \eqref{eq:image-adjoint-kernel}.
\end{proof}

We use row coordinates for the regular right action. If $M_P$ is its
matrix, then
\[
 (M_P)_{\sigma,\sigma\tau}=c_\tau,\qquad
 M_{PQ}=M_PM_Q,\qquad M_{P^\dagger}=M_P^{\mathsf T}.
\]
The involution $\dagger$ on $\Q[S_n]$ differs from the reversal
anti-automorphism of $A$,
\[
 (x_1\cdots x_n)^*=x_n\cdots x_1.
\]
Reversal fixes every Jordan element. On $A_n$, it is right multiplication
by $w_0=[n,n-1,\ldots,1]\in S_n$. We use $I$ for the identity of the
group algebra and for the identity operator.

For a partition $\lambda\vdash n$, let $S^\lambda$ be the rational
Specht module and put $d_\lambda=\dim S^\lambda$. These modules are
absolutely irreducible and form a complete set of irreducible
$\Q[S_n]$-modules \cite{James,JamesKerber}. We write
$\rho_\lambda\colon\Q[S_n]\to\End_\Q(S^\lambda)$ for a matrix
representation acting on column vectors. Thus
$\rho_\lambda(PQ)=\rho_\lambda(P)\rho_\lambda(Q)$. The row convention
for $V_n$ and the column convention for $S^\lambda$ concern different
spaces. We relate them explicitly in Section~\ref{sec:representation}.
Throughout, $\alpha$ denotes a composition and $\lambda$ a partition.

\section{Robbins' recursion and the operators \texorpdfstring{$T_\alpha$}{T\_alpha}}
\label{sec:recursion}

We use Robbins' recursion \cite[Theorem~3(iv)]{Robbins} in the following
form. It applies to the unital Jordan subalgebra generated by $X$ in the
free associative algebra over $\Q$, with the symmetrized product used
here.

\begin{proposition}[Robbins]
\label{prop:Robbins}
For every $m\ge0$,
\begin{equation}
 J_{m+2}=J_{m+1}\cdot J_1+J_m\cdot J_2.
 \label{eq:Robbins}
\end{equation}
\end{proposition}

The recursion controls the factors needed in a spanning family:
at each step, an element of degree $1$ or $2$ is appended on the right.
Induction on the degree therefore gives the following consequence.

\begin{corollary}
\label{cor:linear-quadratic-spanning}
The algebra $J$ is spanned by $1$ and the left-normed Jordan products
whose factors have the form $x$ or $x\cdot y$, where $x,y\in X$.
\end{corollary}

Indeed, $J_1$ is spanned by $X$ and $J_2$ by the products $x\cdot y$;
substituting these spanning sets into \eqref{eq:Robbins} proves the
claim. Compositions with parts $1$ and $2$ now arise naturally: they
record the degrees of the successive factors. We encode each such
pattern by an operator on associative words.

For $n\ge1$, let
\[
 C_n=\{(\alpha_1,\ldots,\alpha_k):
 \alpha_i\in\{1,2\},\ \alpha_1+\cdots+\alpha_k=n\}.
\]
The cardinality of $C_n$ is $F_{n+1}$, where $F_1=F_2=1$ are the
initial Fibonacci numbers. For $\alpha\in C_n$, divide a word $w=x_1\cdots x_n$ into consecutive
blocks of lengths $\alpha_1,\ldots,\alpha_k$. Replace every two-letter
block $xy$ by $x\cdot y$ and take the left-normed Jordan product of all
the blocks. We denote the result by $wT_\alpha$. For example,
\[
 (x_1x_2x_3x_4x_5)T_{(2,1,2)}
 =((x_1\cdot x_2)\cdot x_3)\cdot(x_4\cdot x_5),
\]
\[
 (x_1x_2x_3x_4x_5)T_{(1,2,1,1)}
 =((x_1\cdot(x_2\cdot x_3))\cdot x_4)\cdot x_5.
\]
On expanding, we obtain a linear combination of position permutations
with rational coefficients independent of the letters. Hence
$T_\alpha\in\Q[S_n]$. Every coefficient is a multiple of
$2^{-(n-1)}$, and their sum is $1$.

\begin{corollary}
\label{cor:T-spanning}
For every $n\ge1$,
\begin{equation}
 J_n=\sum_{\alpha\in C_n}A_nT_\alpha,\qquad
 \SJ(n)=\sum_{\alpha\in C_n}V_nT_\alpha.
 \label{eq:T-spanning}
\end{equation}
\end{corollary}

\begin{proof}
By definition, $A_nT_\alpha\subseteq J_n$. Conversely, every product of
total degree $n$ in Corollary~\ref{cor:linear-quadratic-spanning} is
$wT_\alpha$: take $\alpha$ to be the sequence of factor degrees and
read the letters of the factors in order to obtain $w$. This proves the
first equality.

The algebra $J$ is graded by multidegree, since each Jordan monomial
is multihomogeneous. Projection onto the component in which each of
$z_1,\ldots,z_n$ occurs once maps $J_n$ onto $\SJ(n)$ and commutes with
every position permutation. Applying this projection to the first
equality gives the second.
\end{proof}

\begin{remark}
Different compositions may give the same operator. For instance,
$T_{(2,1)}=T_{(1,1,1)}$. More generally, replacing an initial part $2$
by $1,1$ does not change $T_\alpha$. We keep every composition in
$C_n$, including these repetitions. Removing repetitions changes the
Gram operator and its characteristic polynomial, although its image
and the orthogonal projection onto that image remain the same.
\end{remark}

\section{The Gram operator and the polynomial criterion}
\label{sec:criterion}

By Corollary~\ref{cor:T-spanning}, the multilinear Jordan space
is the sum of the images of the operators
$R_{T_\alpha}$:
\[
SJ(n)=\sum_{\alpha\in C_n}V_nT_\alpha.
\]
This description gives a spanning set, but
does not yet provide a convenient membership
criterion. We would like to replace this sum
of images by the image of a single operator.

The adjoint relation from Lemma~\ref{lem:group-algebra-adjoint} suggests
how to do this. In a finite-dimensional inner
product space, a subspace is determined by its
orthogonal complement, and the orthogonal
complement of a sum is the intersection of
the orthogonal complements. Therefore, instead
of working directly with the images
$V_nT_\alpha$, we can first study the
intersection of the kernels of their adjoints.

This leads naturally to the Gram-type element
\begin{equation}
 U_n=\sum_{\alpha\in C_n}T_\alpha^\dagger T_\alpha\in\Q[S_n].
 \label{eq:U-definition}
\end{equation}
The order of the factors is important:
our operators act on the right, so that
\[
fU_n=\sum_{\alpha\in C_n}
(fT_\alpha^\dagger)T_\alpha.
\]
As we now show, the kernel of $R_{U_n}$
is precisely the intersection of the kernels
of $R_{T_\alpha^\dagger}$, and its image
is the sum appearing above.

\begin{lemma}
\label{lem:Gram-image}
The operator $R_{U_n}$ is self-adjoint and positive semidefinite after
extension of scalars to $\R$. Moreover,
\begin{equation}
 \ker R_{U_n}=\bigcap_{\alpha\in C_n}\ker R_{T_\alpha^\dagger},
 \qquad V_nU_n=\SJ(n).
 \label{eq:Gram-image}
\end{equation}
The kernel equality holds over both $\Q$ and $\R$.
\end{lemma}

\begin{proof}
By Lemma~\ref{lem:group-algebra-adjoint},
\[
U_n^\dagger
 =\sum_{\alpha\in C_n}
   (T_\alpha^\dagger T_\alpha)^\dagger
 =\sum_{\alpha\in C_n}
   T_\alpha^\dagger T_\alpha
 =U_n.
\]
Thus $R_{U_n}$ is self-adjoint. Moreover, for every $f\in V_n$,
\[
\begin{aligned}
\langle fU_n,f\rangle
&=\sum_{\alpha\in C_n}
  \langle fT_\alpha^\dagger T_\alpha,f\rangle\\
&=\sum_{\alpha\in C_n}
  \langle fT_\alpha^\dagger,
          fT_\alpha^\dagger\rangle.
\end{aligned}
\]
The same equality holds after extending
scalars to $\mathbb R$, where every term
on the right is nonnegative. Hence
$R_{U_n}$ is positive semidefinite.

Since the word basis is orthonormal, the
sum on the right vanishes if and only if
$fT_\alpha^\dagger=0$ for every $\alpha$.
It follows that
\[
\ker R_{U_n}
=\bigcap_{\alpha\in C_n}
  \ker R_{T_\alpha^\dagger}.
\]
Indeed, if $fU_n=0$, the displayed
sum of squared norms is zero, while the
converse follows immediately from the
definition of $U_n$.

We now determine the image of $R_{U_n}$.
Since this operator is self-adjoint,
\eqref{eq:image-adjoint-kernel}  in Lemma~\ref{lem:group-algebra-adjoint} and the finite-dimensional
orthogonal-complement identities give
\[
\begin{aligned}
(V_nU_n)^\perp
&=\ker R_{U_n}=\bigcap_{\alpha\in C_n}
  \ker R_{T_\alpha^\dagger}\\
&=\bigcap_{\alpha\in C_n}
  (V_nT_\alpha)^\perp=\left(
   \sum_{\alpha\in C_n}V_nT_\alpha
  \right)^\perp.
\end{aligned}
\]
Taking orthogonal complements once more
and using Corollary~\ref{cor:T-spanning}, we obtain
\[
V_nU_n
=\sum_{\alpha\in C_n}V_nT_\alpha
=SJ(n).
\]
All the identities of subspaces above hold
over $\mathbb Q$. Thus no extension of
scalars is needed to establish the image
and kernel formulas.
\end{proof}

In particular, $\ker R_{U_n}=\SJ(n)^\perp$. Notice that the kernel
identity involves the adjoints $T_\alpha^\dagger$; this follows from
the chosen right action.

We write the characteristic polynomial in the form
\begin{equation}
 \chi_n(t)=\det(tI-R_{U_n})=t^{e_n}q_n(t),\qquad q_n(0)\ne0.
 \label{eq:characteristic-factor}
\end{equation}
The polynomial $q_n$ retains the multiplicity of every nonzero root of
$\chi_n$.

By Lemma~\ref{lem:Gram-image}, $R_{U_n}$ is self-adjoint and positive
semidefinite on the finite-dimensional real inner product space
$V_n\otimes_\Q\R$. The real spectral theorem
\cite[Section~7B]{Axler2024} therefore makes it diagonalizable with
nonnegative eigenvalues. In particular, the algebraic multiplicity
of zero is its nullity. Extension of scalars preserves the rank of a
rational matrix, so
\begin{equation}
 e_n=\dim_\Q\ker R_{U_n},\qquad \deg q_n=n!-e_n.
 \label{eq:e_n}
\end{equation}
The next identity is what allows the characteristic polynomial to
detect the image of $R_{U_n}$.

\begin{lemma}
\label{lem:annihilating-polynomial}
We have $U_nq_n(U_n)=0$ in $\Q[S_n]$.
\end{lemma}

\begin{proof}
The polynomial $tq_n(t)$ vanishes at every eigenvalue of $R_{U_n}$:
the factor $t$ accounts for zero and $q_n$ for every nonzero
eigenvalue. Diagonalizability gives $R_{U_n}q_n(R_{U_n})=0$ over
$\R$. The matrix and polynomial have rational coefficients, so the
same identity holds over $\Q$. Faithfulness of the right regular
action yields the identity in $\Q[S_n]$.
\end{proof}

We can now explain the polynomial criterion.
By Lemma~\ref{lem:Gram-image}, $SJ(n)$ is the image of
$R_{U_n}$. The spectral theorem shows that
$q_n(R_{U_n})$ vanishes on this image,
whereas on $\ker R_{U_n}$ it acts as
multiplication by the nonzero scalar
$q_n(0)$.

Thus
\[
I-\frac{q_n(R_{U_n})}{q_n(0)}
\]
is the orthogonal projection onto
$SJ(n)\otimes_{\mathbb Q}\mathbb R$.
The following proof establishes the
corresponding criterion over $\mathbb Q$
directly from Lemmas~\ref{lem:Gram-image} and~\ref{lem:annihilating-polynomial}.

\begin{proof}[Proof of Theorem~\ref{thm:A}]
The assertions about the image, self-adjointness and positivity follow
from Lemma~\ref{lem:Gram-image}, and \eqref{eq:e_n} gives the nullity.
Write
\[
 q_n(t)=q_n(0)+t h_n(t),\qquad h_n(t)\in\Q[t].
\]
If $f\in\SJ(n)$, then $f=bU_n$ for some $b\in V_n$, and
Lemma~\ref{lem:annihilating-polynomial} gives $fq_n(U_n)=0$.
Conversely, if $fq_n(U_n)=0$, then
\[
 f=-\frac{fh_n(U_n)}{q_n(0)}U_n\in V_nU_n=\SJ(n).
\]
Thus $f\in\SJ(n)$ if and only if $fq_n(U_n)=0$.

Put
\begin{equation}
 \Pi_n=I-\frac{q_n(U_n)}{q_n(0)}
      =-\frac{U_nh_n(U_n)}{q_n(0)}.
 \label{eq:projection}
\end{equation}
It follows that $V_n\Pi_n\subseteq\SJ(n)$ and that $f\Pi_n=f$ for
every $f\in\SJ(n)$. Hence $V_n\Pi_n=\SJ(n)$, and $f\Pi_n=f$ holds
exactly for the elements of this space. Applying the projection twice
therefore gives the same operator. Faithfulness of the right regular
action yields $\Pi_n^2=\Pi_n$ in $\Q[S_n]$.
\end{proof}

\begin{corollary}
\label{cor:projection-dimensions}
We have $V_n=\SJ(n)\oplus\ker R_{U_n}$,
$\ker R_{\Pi_n}=\ker R_{U_n}$, and
\begin{equation}
 \dim\SJ(n)=\rank R_{\Pi_n}=\tr R_{\Pi_n}=n!-e_n=\deg q_n.
 \label{eq:projection-dimensions}
\end{equation}
Over $\R$, $R_{\Pi_n}$ is the orthogonal projection onto
$\SJ(n)\otimes_\Q\R$.
\end{corollary}

\begin{proof}
The identity $U_nq_n(U_n)=0$ gives
$U_n\Pi_n=\Pi_nU_n=U_n$. Thus $f\Pi_n=0$ implies $fU_n=0$.
The converse follows from the second expression in
\eqref{eq:projection}. Since $R_{\Pi_n}$ is idempotent, its image and
kernel form a direct sum and its rank equals its trace. The dimension
formula follows from \eqref{eq:e_n}. Finally, $\Pi_n$ is a polynomial
in the self-adjoint element $U_n$, so the projection is self-adjoint.
\end{proof}

\begin{corollary}[Homogeneous Jordan criterion]
\label{cor:homogeneous}
For every $n\ge1$, we have $A_nU_n=A_n\Pi_n=J_n$. For $f\in A_n$, not
necessarily multilinear,
\[
 f\in J_n\quad\Longleftrightarrow\quad fq_n(U_n)=0
 \quad\Longleftrightarrow\quad f\Pi_n=f.
\]
Consequently, $f=\sum_{n\ge0}f_n\in A$, with $f_n\in A_n$, belongs to
$J$ if and only if $f_nq_n(U_n)=0$ for every $n\ge1$.
\end{corollary}

\begin{proof}
By Theorem~\ref{thm:A}, the operator $R_{\Pi_n}$ fixes $V_nT_\alpha$.
Faithfulness of the regular action therefore gives
\[
 T_\alpha\Pi_n=T_\alpha\qquad\text{in }\Q[S_n].
\]
This identity also acts on $A_n$. Using \eqref{eq:T-spanning}, we
obtain $f\Pi_n=f$ for every $f\in J_n$. On the other hand,
\eqref{eq:U-definition} and \eqref{eq:projection} give
\[
 A_n\Pi_n\subseteq A_nU_n
 \subseteq\sum_{\alpha\in C_n}A_nT_\alpha=J_n.
\]
These inclusions, together with $J_n\Pi_n=J_n$, prove both image
equalities. The two equivalences follow from the definition of
$\Pi_n$. Finally, $J$ is graded and $J_0=\Q1$, so the degree-zero
component imposes no condition.
\end{proof}

The proof uses only identities in $\Q[S_n]$ after the multilinear
criterion has been established. It therefore applies to an arbitrary
set $X$, finite or infinite. The dimension identity
$\deg q_n=\dim\SJ(n)$ concerns the regular multilinear component; it
is not a formula for $\dim J_n$ for an arbitrary generating set.

\begin{remark}
\label{rem:alternative-polynomial}
The rank and trace of $U_n$ need not coincide. For example,
$\tr R_{U_3}=9/2$ and $\rank R_{U_3}=3$. Also, $q_n$ is the nonzero
factor of the characteristic polynomial, not of the minimal polynomial.
If $p_n(0)\ne0$ and $p_n$ vanishes at every nonzero eigenvalue of
$R_{U_n}$, then
\[
 I-\frac{p_n(U_n)}{p_n(0)}=\Pi_n.
\]
Indeed, the two operators agree on a real eigenbasis, and the regular
action is faithful. Thus the nonzero factor of the minimal polynomial
can be used to evaluate $\Pi_n$, but its degree cannot replace
$\deg q_n$ in \eqref{eq:projection-dimensions}.
\end{remark}

Omitting the unit from the generators of $J$ changes neither $J_n$ for
$n\ge1$ nor the criterion in positive degree, since $1\cdot a=a$.

\section{Representation theory of \texorpdfstring{$U_n$}{U\_n}}
\label{sec:representation}

Maschke's theorem applies because $\operatorname{char}\Q=0$.
Moreover, $\Q$ is a splitting field for $S_n$, with the absolutely
irreducible Specht modules as its simple modules. The
Artin--Wedderburn theorem gives the algebra isomorphism
\cite{James,JamesKerber}
\begin{equation}
 \Q[S_n]\longrightarrow
 \bigoplus_{\lambda\vdash n}\End_\Q(S^\lambda),\qquad
 P\longmapsto(\rho_\lambda(P))_\lambda
 \label{eq:Wedderburn}
\end{equation}
Under this map,
left multiplication by $g$ and right multiplication by $h$ act on the
$\lambda$-summand by
\[
 g\cdot X\cdot h=\rho_\lambda(g)X\rho_\lambda(h).
\]
Identify the endomorphism $w\mapsto\varphi(w)v$ with
$v\otimes\varphi\in S^\lambda\otimes(S^\lambda)^*$. The left action
is on the first factor. Right multiplication by $P$ has the form
\begin{equation}
 v\otimes\varphi\longmapsto
 v\otimes(\varphi\circ\rho_\lambda(P)).
 \label{eq:multiplicity-action}
\end{equation}
Thus $(S^\lambda)^*$ is the multiplicity space for the left action.

\begin{proof}[Proof of Theorem~\ref{thm:B}]
Put $M_\lambda=\rho_\lambda(U_n)$ and let
$\psi_\lambda(\varphi)=\varphi\circ M_\lambda$. If the elements of
the dual space are written as columns in the dual basis, then
$\psi_\lambda$ has matrix $M_\lambda^{\mathsf T}$. Indeed,
\[
 (\varphi\circ M_\lambda)(e_i)
 =\sum_j(M_\lambda)_{ji}\varphi(e_j).
\]
Hence $\rank\psi_\lambda=\rank M_\lambda=r_\lambda$.

By \eqref{eq:multiplicity-action}, the image of right multiplication
by $U_n$ on the $\lambda$-summand is
$S^\lambda\otimes\im\psi_\lambda$. As a left module, this is a sum
of $r_\lambda$ copies of $S^\lambda$. Theorem~\ref{thm:A} now gives
\[
 \SJ(n)\cong\bigoplus_{\lambda\vdash n}r_\lambda S^\lambda.
\]
The image and kernel on the $\lambda$-summand have dimensions
$d_\lambda r_\lambda$ and $d_\lambda(d_\lambda-r_\lambda)$,
respectively. Summing gives the two dimension formulas.

Finally, the matrix on this summand consists of $d_\lambda$ copies of
$M_\lambda^{\mathsf T}$. Transposition preserves the characteristic
polynomial, so
\[
 \chi_n(t)=\prod_{\lambda\vdash n}
           \det(tI-M_\lambda)^{d_\lambda}.
\]
This proves the last assertion.
\end{proof}

The precomposition in \eqref{eq:multiplicity-action} is a right
action. It does not use the inverse that occurs in the usual left
contragredient action on $(S^\lambda)^*$.

To compute the blocks, we also need the form of the adjoint in a
basis which is not orthonormal.

\begin{lemma}
\label{lem:invariant-form}
For each $\lambda$, there is a positive definite rational symmetric
matrix $H_\lambda$ such that
\begin{equation}
 \rho_\lambda(P^\dagger)
 =H_\lambda^{-1}\rho_\lambda(P)^{\mathsf T}H_\lambda
 \qquad(P\in\Q[S_n]).
 \label{eq:block-adjoint}
\end{equation}
If $A_\alpha=\rho_\lambda(T_\alpha)$, then
\begin{equation}
 \rho_\lambda(U_n)
 =H_\lambda^{-1}\sum_{\alpha\in C_n}
        A_\alpha^{\mathsf T}H_\lambda A_\alpha,
 \qquad
 \Row\rho_\lambda(U_n)=\sum_{\alpha\in C_n}\Row A_\alpha.
 \label{eq:block-Gram}
\end{equation}
In particular, every block is diagonalizable over $\R$ and has
nonnegative eigenvalues.
\end{lemma}

\begin{proof}
Averaging a rational positive definite form over $S_n$ gives an
invariant form. For example, we may take
\[
 H_\lambda=\sum_{g\in S_n}
       \rho_\lambda(g)^{\mathsf T}\rho_\lambda(g).
\]
Then $\rho_\lambda(g)^{\mathsf T}H_\lambda\rho_\lambda(g)=H_\lambda$,
which gives \eqref{eq:block-adjoint} for $g$ and, by linearity, for $P$.
The first formula in \eqref{eq:block-Gram} follows.

Let $B=\sum_\alpha A_\alpha^{\mathsf T}H_\lambda A_\alpha$.
For a real column vector $v$,
\[
 v^{\mathsf T}Bv
 =\sum_\alpha(A_\alpha v)^{\mathsf T}H_\lambda(A_\alpha v).
\]
Positivity implies $\ker B=\bigcap_\alpha\ker A_\alpha$.
The row space of a matrix is the annihilator of its column kernel;
thus $\Row B=\sum_\alpha\Row A_\alpha$. Multiplication by the
invertible matrix $H_\lambda^{-1}$ does not change the row space.
These equalities hold over $\Q$ as well, since the matrices are
rational and ranks are unchanged by extension of scalars.
Finally, $H_\lambda^{-1}B$ is similar over $\R$ to the symmetric
positive semidefinite matrix $H_\lambda^{-1/2}BH_\lambda^{-1/2}$.
\end{proof}

\begin{corollary}
\label{cor:block-characteristic}
For each $\lambda\vdash n$, write
\[
 \det(tI-\rho_\lambda(U_n))
 =t^{d_\lambda-r_\lambda}q_\lambda(t),\qquad q_\lambda(0)\ne0.
\]
Then
\begin{equation}
\begin{split}
 \chi_n(t)&=\prod_{\lambda\vdash n}
    \bigl(t^{d_\lambda-r_\lambda}q_\lambda(t)\bigr)^{d_\lambda},\\
 q_n(t)&=\prod_{\lambda\vdash n}q_\lambda(t)^{d_\lambda}.
\end{split}
 \label{eq:block-characteristic}
\end{equation}
\end{corollary}

\begin{proof}
By Lemma~\ref{lem:invariant-form}, the zero eigenvalue of each block
has algebraic multiplicity $d_\lambda-r_\lambda$. Apply
Theorem~\ref{thm:B} and separate the powers of $t$.
\end{proof}

\subsection*{Computing the multiplicities}

We now give a recursion which computes the row spaces of the blocks
without forming $U_n$. It is the multilinear form of Robbins'
recursion in the Specht modules.

We use rational seminormal representations, with bases indexed by
standard Young tableaux \cite{JamesKerber,OkounkovVershik}. We specify
the normalization. If $c_T(j)$ is the column index minus the row index
of the box occupied by $j$, put
$a_i(T)=c_T(i+1)-c_T(i)$. For $s_i=(i,i+1)$, we take
\begin{equation}
 \rho_\lambda(s_i)e_T=
 \begin{cases}
   a_i(T)^{-1}e_T,& |a_i(T)|=1,\\
   a_i(T)^{-1}e_T+e_{s_iT},& a_i(T)>1,\\
   a_i(T)^{-1}e_T+(1-a_i(T)^{-2})e_{s_iT},& a_i(T)<-1.
 \end{cases}
 \label{eq:seminormal}
\end{equation}
Here $s_iT$ is standard in the last two cases. The matrices are
rational but need not be symmetric, so \eqref{eq:block-adjoint} is
required when forming their adjoints.

The basis is compatible with the chain
$S_1\subset S_2\subset\cdots\subset S_n$. For $k=1,2$ with $k\le n$,
group the tableaux by the positions of $n-k+1,\ldots,n$. For a group
$G$, removing these entries leaves a shape $\mu(G)\vdash n-k$ and
identifies $G$ with $\SYT(\mu(G))$. The restriction of
$\rho_\lambda(f)$, for $f\in\Q[S_{n-k}]$, is block diagonal, with
block $\rho_{\mu(G)}(f)$ on $G$. This follows directly from
\eqref{eq:seminormal}, since the generators of $S_{n-k}$ only move
entries at most $n-k$.

Put $W_\lambda=\Row\rho_\lambda(U_n)$ and
$W_\varnothing=W_{(1)}=\Q$. For row vectors, let $\iota_G$ be the
coordinate embedding obtained from the bijection
$\SYT(\mu(G))\to G$, and define
\[
 E^{(k)}_\lambda=\bigoplus_G\iota_G(W_{\mu(G)})\qquad(k=1,2).
\]
Different groups have disjoint coordinate supports. We also put
\[
 c_n=s_{n-1}s_{n-2}\cdots s_1=[n,1,2,\ldots,n-1].
\]
Under the right position action, $c_n$ moves the last letter of a word
to the front.

\begin{proposition}
\label{prop:row-recursion-main}
For $n\ge2$ and $\lambda\vdash n$,
\begin{equation}
\begin{split}
 W_\lambda={}&E^{(1)}_\lambda(I+\rho_\lambda(c_n))\\
 &+E^{(2)}_\lambda(I+\rho_\lambda(s_{n-1}))
                         (I+\rho_\lambda(c_n)^2).
\end{split}
 \label{eq:row-recursion}
\end{equation}
Starting with $W_\varnothing=W_{(1)}=\Q$, exact row reduction gives
$r_\lambda=\dim W_\lambda$ in every degree.
\end{proposition}

\begin{proof}
Under \eqref{eq:Wedderburn}, the left ideal
$\SJ(n)=\Q[S_n]U_n$ maps to
\[
 \bigoplus_{\lambda\vdash n}
 \End_\Q(S^\lambda)\rho_\lambda(U_n).
\]
The $\lambda$-summand consists of all matrices whose rows belong to
$W_\lambda$. Therefore
\begin{equation}
 \sum_{f\in\SJ(n)}\Row\rho_\lambda(f)=W_\lambda.
 \label{eq:Jordan-row-space}
\end{equation}

Taking the multilinear component of \eqref{eq:Robbins}, we see that
$\SJ(n)$ is spanned by the relabellings of
$f\cdot z_n$ and $h\cdot(z_{n-1}\cdot z_n)$, where
$f\in\SJ(n-1)$ and $h\in\SJ(n-2)$ use the first $n-1$ and $n-2$
variables, respectively. We set $\SJ(0)=\Q1$.

Embed $S_{n-1}$ and $S_{n-2}$ in $S_n$ by fixing the last positions.
Under the word identification, appending $z_n$ to a word from $V_{n-1}$
is this embedding, and moving $z_n$ to the front is multiplication by
$c_n$. Hence
\[
 f\cdot z_n=\tfrac12 f(I+c_n).
\]
Likewise, $\tfrac12(I+s_{n-1})$ symmetrizes the final two letters, and
$c_n^2=[n-1,n,1,\ldots,n-2]$ moves them to the front. Thus
\[
 h\cdot(z_{n-1}\cdot z_n)
 =\tfrac14h(I+s_{n-1})(I+c_n^2).
\]
For $n=2$, this formula remains valid with $h=1$ and $c_2^2=I$.
Since nonzero scalar factors do not affect a span, Robbins' recursion
becomes
\[
\begin{split}
 \SJ(n)=\sum_{g\in S_n}g\bigl(&\SJ(n-1)(I+c_n)\\
 &+\SJ(n-2)(I+s_{n-1})(I+c_n^2)\bigr).
\end{split}
\]

Left multiplication by $\rho_\lambda(g)$ preserves the row space,
and $\Row(BD)=\Row(B)D$ for any matrices $B,D$ of compatible sizes.
By \eqref{eq:Jordan-row-space}, it remains to identify
\[
 \sum_{f\in\SJ(n-k)}\Row\rho_\lambda(f)\qquad(k=1,2).
\]
The branching decomposition gives
\[
 \Row\rho_\lambda(f)
 =\bigoplus_G\iota_G(\Row\rho_{\mu(G)}(f)).
\]
Each embedded row is supported in its own group, even if the same
shape occurs in several groups. Taking the sum over $f$ and using
\eqref{eq:Jordan-row-space} in degree $n-k$ therefore gives
\[
 \sum_{f\in\SJ(n-k)}\Row\rho_\lambda(f)
 =\bigoplus_G\iota_G(W_{\mu(G)})=E^{(k)}_\lambda.
\]
For $n=k=2$, this uses $W_\varnothing=\Q$. Substitution proves
\eqref{eq:row-recursion}.
\end{proof}

The recursion proceeds degree by degree. We embed the row bases from
degrees $n-1$ and $n-2$, apply the matrices in
\eqref{eq:row-recursion}, and take a basis of their sum. Each row has
$d_\lambda$ coordinates; row reduction may use a rectangular stacked
matrix. No matrix of order $n!$ is needed. The dimension of the resulting
space is $r_\lambda$, so Theorem~\ref{thm:B} gives the dimension and
nullity in degree $n$. The computational details and the verified range
are recorded in  \ref{app:computations}.

\section{Examples and computed dimensions}
\label{sec:examples}

\subsection{Degree three directly}

$C_3$ consists of $(1,1,1)$, $(2,1)$ and $(1,2)$. Write $[ijk]$ for the
permutation with one-line notation $(i,j,k)$ and, in this subsection
only, $ijk$ for the word $z_iz_jz_k$. Put $P=T_{(2,1)}=T_{(1,1,1)}$ and
$Q=T_{(1,2)}$; then
\[
\begin{aligned}
 123P&=(z_1\cdot z_2)\cdot z_3=\tfrac14(123+213+312+321),\\
 123Q&=z_1\cdot(z_2\cdot z_3)=\tfrac14(123+132+231+321),
\end{aligned}
\]
so $P=\tfrac14([123]+[213]+[312]+[321])$ and
$Q=\tfrac14([123]+[132]+[231]+[321])$. Inverting permutations,
$P^\dagger=\tfrac14([123]+[213]+[231]+[321])$, and multiplication in
$\Q[S_3]$ gives $P^\dagger P=Q^\dagger Q$. Keeping all three compositions,
\begin{equation}
 U_3=2P^\dagger P+Q^\dagger Q=3P^\dagger P
 =\tfrac34\bigl([123]+[321]\bigr)+\tfrac38\bigl([132]+[213]+[231]+[312]\bigr).
 \label{eq:U3}
\end{equation}
The matrix $M_{U_3}$ is symmetric of order $6$, with rank $3$ and
characteristic polynomial $t^3q_3(t)$, where
\[
 q_3(t)=(t-3)\Bigl(t-\frac34\Bigr)^2=t^3-\frac92t^2+\frac{81}{16}t-\frac{27}{16};
\]
we recover this below from the irreducible blocks. By
\eqref{eq:projection},
\begin{equation}
 \Pi_3=I-\frac{q_3(U_3)}{q_3(0)}=3U_3-\frac83U_3^2+\frac{16}{27}U_3^3 .
 \label{eq:Pi3}
\end{equation}
The characteristic factor retains repeated roots. For the projection,
Remark~\ref{rem:alternative-polynomial} allows us to use instead
$p_3(t)=(t-3)(t-\tfrac34)$, giving
\[
 \Pi_3=\frac53U_3-\frac49U_3^2.
\]
Multiplication in \eqref{eq:U3} gives
\[
 U_3^2=\frac{15}{4}U_3-\frac98(I+[321]),
\]
and consequently
\begin{equation}
 \Pi_3=\frac12(I+[321]).
 \label{eq:Pi3-reversal}
\end{equation}
Thus the general construction recovers the averaging operator
$f\mapsto(f+f^*)/2$ in degree three. Its image is the space of
reversible multilinear elements, so $\SJ(3)$ is precisely this space.
Idempotence also follows from $[321]^2=I$. Notice that
$\tr R_{U_3}=9/2$ whereas $\rank R_{U_3}=3$; these quantities agree
for $\Pi_3$, but not for $U_3$.

For example, $f_1=123+321$ is fixed by $\Pi_3$ and hence is a Jordan
element. Its explicit Jordan expression is
\[
 f_1=2\bigl((z_1\cdot z_2)\cdot z_3
 +(z_2\cdot z_3)\cdot z_1
 -(z_1\cdot z_3)\cdot z_2\bigr),
\]
as follows by expanding the three products.

\subsection{Degree three through the blocks}

The partitions of $3$ are $(3)$, $(2,1)$, $(1,1,1)$ with
$d_\lambda=1,2,1$. For $(2,1)$ order the tableaux as
\[
 T_1=\begin{array}{@{}cc@{}}1&3\\2&\end{array},\qquad
 T_2=\begin{array}{@{}cc@{}}1&2\\3&\end{array}.
\]
In the seminormal form \eqref{eq:seminormal}, the adjacent
transpositions act by
\[
 \rho_{(2,1)}(s_1)=\begin{pmatrix}-1&0\\0&1\end{pmatrix},\qquad
 \rho_{(2,1)}(s_2)=\begin{pmatrix}\tfrac12&\tfrac34\\[2pt]1&-\tfrac12\end{pmatrix},
\]
with the normalization specified in Section~\ref{sec:representation}.
The basis is not orthonormal. Its invariant form may be taken to be
$H_{(2,1)}=\operatorname{diag}(1,3/4)$, so the adjoint is given by
\eqref{eq:block-adjoint}.
Substituting the six permutations of \eqref{eq:U3} gives
\begin{equation}
 \rho_{(3)}(U_3)=\begin{pmatrix}3\end{pmatrix},\qquad
 \rho_{(2,1)}(U_3)=\begin{pmatrix}\tfrac9{16}&-\tfrac9{32}\\[2pt]-\tfrac38&\tfrac3{16}\end{pmatrix},\qquad
 \rho_{(1,1,1)}(U_3)=\begin{pmatrix}0\end{pmatrix}.
 \label{eq:U3-blocks}
\end{equation}
For the trivial and sign representations these are the sum of the
coefficients in \eqref{eq:U3} and the signed sum, respectively. The
middle matrix has determinant $0$ and trace $\tfrac34$, hence rank $1$
and characteristic polynomial $t(t-\tfrac34)$. Theorem~\ref{thm:B} gives
\[
 \SJ(3)\cong S^{(3)}\oplus S^{(2,1)},\qquad
 \dim\SJ(3)=1+2+0=3,\qquad
 e_3=0+2+1=3,
\]
and Corollary~\ref{cor:block-characteristic} gives
\[
 \chi_3(t)=(t-3)^1\bigl(t(t-\tfrac34)\bigr)^2t^1=t^3(t-3)(t-\tfrac34)^2=t^3q_3(t),
\]
in agreement with the direct computation.

\subsection{Degree four}

The exact ranks in degree four are
\begin{center}
\begin{tabular}{c|ccccc}
 $\lambda$ & $(4)$ & $(3,1)$ & $(2,2)$ & $(2,1,1)$ & $(1,1,1,1)$\\
 \hline
 $d_\lambda$ & $1$ & $3$ & $2$ & $3$ & $1$\\
 $r_\lambda$ & $1$ & $1$ & $2$ & $1$ & $0$
\end{tabular}
\end{center}
so $\SJ(4)\cong S^{(4)}\oplus S^{(3,1)}\oplus2S^{(2,2)}\oplus S^{(2,1,1)}$,
$\dim\SJ(4)=1+3+4+3+0=11$ and $e_4=0+6+0+6+1=13$. The block
characteristic polynomials are $t-5$, $t^2(t-\tfrac98)$,
$(t-\tfrac12)(t-\tfrac94)$, $t^2(t-\tfrac38)$ and $t$, so
\[
 q_4(t)=(t-5)\Bigl(t-\frac12\Bigr)^2\Bigl(t-\frac94\Bigr)^2\Bigl(t-\frac98\Bigr)^3\Bigl(t-\frac38\Bigr)^3 .
\]
The degree of $q_4$ is $11$, in agreement with
\eqref{eq:projection-dimensions}. Direct computation with the
$24\times24$ matrix $M_{U_4}$ gives the same characteristic polynomial. The
multiplicity $2$ of $S^{(2,2)}$ is what distinguishes $\SJ(4)$ from the
simple Jordan elements: the image of $T_{(1,1,1,1)}$ is the module
$S^{(4)}\oplus S^{(3,1)}\oplus S^{(2,2)}\oplus S^{(2,1,1)}$ of dimension
$9$, computed in the same way from the ranks of $\rho_\lambda(T_{(1,1,1,1)})$,
and $(z_1\cdot z_2)\cdot(z_3\cdot z_4)$ is a Jordan element that is not
simple. We give the factorization of $q_5$ in  \ref{app:q5}.
The block polynomials through degree $7$ are included in the
verification files described in  \ref{app:computations}.

\subsection{Dimensions, block sizes and comparison with the free Jordan algebra}

We apply Proposition~\ref{prop:row-recursion-main} degree by degree.
Table~\ref{tab:dimensions} gives the dimensions obtained by exact
rational row reduction through degree $12$. The full list of
$d_\lambda$ and $r_\lambda$ covers all $271$ partitions of degrees
$1$ through $12$ and is available in the GitHub
repository~\cite{JordanCode}. The decompositions in degrees $8$
through $12$ are also listed in  \ref{app:multiplicities}.
In particular, the computation gives
\[
 \dim\SJ(11)=19\,565\,150,\qquad
 \dim\SJ(12)=236\,416\,818.
\]
The range of these multiplicity calculations is distinct from the
range of the block characteristic-polynomial calculations, which
extend through degree $7$.

The column $\dim\SJ(n)$ also gives $\deg q_n$. To compare it with
reversibility, let $H$ be the fixed space of the reversal
anti-automorphism $w\mapsto w^*$. Every Jordan element is reversible.
For $n\ge2$, reversal pairs distinct multilinear words, so
\[
 \SJ(n)\subseteq H\cap V_n,\qquad \dim(H\cap V_n)=n!/2.
\]
By Cohn's theorem \cite[Theorem~4.2]{Cohn}, every reversible element in at most three
variables is a Jordan element. Thus equality holds in degrees $2$ and
$3$. In degree $4$, the tetrad
$z_1z_2z_3z_4+z_4z_3z_2z_1$ is reversible but is not a Jordan element.
The last column records the difference of these dimensions.

\begin{table}[htbp]
\centering
\caption{Multilinear dimensions and nullities from exact rational row reduction, with maximal Specht-block orders.}
\label{tab:dimensions}
\small
\setlength{\tabcolsep}{4pt}
\begin{tabular}{rrrrrr}
\toprule
 $n$ & $n!$ & $\max_{\lambda\vdash n}d_\lambda$ & $e_n$ & $\dim\SJ(n)$ & $n!/2-\dim\SJ(n)$\\
\midrule
 1 & 1 & 1 & 0 & 1 & --\\
 2 & 2 & 1 & 1 & 1 & 0\\
 3 & 6 & 2 & 3 & 3 & 0\\
 4 & 24 & 3 & 13 & 11 & 1\\
 5 & 120 & 6 & 65 & 55 & 5\\
 6 & 720 & 16 & 390 & 330 & 30\\
 7 & 5040 & 35 & 2695 & 2345 & 175\\
 8 & 40320 & 90 & 21273 & 19047 & 1113\\
 9 & 362880 & 216 & 188865 & 174015 & 7425\\
 10 & 3628800 & 768 & 1867105 & 1761695 & 52705\\
 11 & 39916800 & 2310 & 20351650 & 19565150 & 393250\\
 12 & 479001600 & 7700 & 242584782 & 236416818 & 3083982\\
\bottomrule
\end{tabular}
\end{table}

We check the dimensions $d_\lambda$ by counting standard tableaux and
by the hook-length formula \cite{Macdonald}; in every degree listed,
$\sum_{\lambda\vdash n}d_\lambda^2=n!$. In degree $9$, the maximal
value $216$ occurs for $(4,3,1,1)$ and $(4,2,2,1)$, whose multiplicities
are $106$ and $108$. In degree $10$, the maximal value $768$ occurs for
$(4,3,2,1)$, whose multiplicity is $383$. In degree $12$, the
maximal value $7700$ occurs for $(5,3,2,1,1)$, whose multiplicity is
$3837$. The separate regular-matrix
calculations in degrees at most $5$ agree with the block calculations;
 \ref{app:computations} gives the extent of these checks.

The multiplicities $r_\lambda$ can be compared with the multiplicities of
$S^\lambda$ in the multilinear components $\mathrm{Jord}(n)$ of the free
Jordan algebra, computed for $n\le10$ by Dotsenko and Hentzel
\cite[Proposition~4.2]{DotsenkoHentzel}. The surjection from the free Jordan algebra onto
the free special Jordan algebra is $S_n$-equivariant, so each $r_\lambda$
is bounded above by the corresponding multiplicity in $\mathrm{Jord}(n)$,
with equality exactly when no multilinear $s$-identity of degree $n$ has
a $\lambda$-component. For $n\le7$, the multiplicities in our calculation agree with their
table. Thus the canonical map from $\mathrm{Jord}(n)$ to $\SJ(n)$ is an
isomorphism in these degrees. In degree $8$ they differ in exactly one
entry: the multiplicity of $S^{(3,3,2)}$ is $18$ in $\SJ(8)$ and $19$ in
$\mathrm{Jord}(8)$, so the multilinear $s$-identities of degree $8$ form
a single copy of $S^{(3,3,2)}$, of dimension $42$, and
$\dim\mathrm{Jord}(8)=19047+42=19089$. 
In degree $9$ they differ in exactly five entries: the multiplicities of
$S^{(4,3,2)}$, $S^{(4,3,1^2)}$, $S^{(4,2^2,1)}$, $S^{(3^3)}$, and
$S^{(3^2,2,1)}$ are, respectively, $84,106,108,18,83$ in $\SJ(9)$ and
$86,107,109,20,85$ in $\mathrm{Jord}(9)$. Thus, the multilinear
$s$-identities of degree $9$ form the module
\[
2S^{(4,3,2)}
\oplus S^{(4,3,1^2)}
\oplus S^{(4,2^2,1)}
\oplus 2S^{(3^3)}
\oplus 2S^{(3^2,2,1)},
\]
of dimension
$2\cdot168+216+216+2\cdot42+2\cdot168=1188$, and
\[
\dim\mathrm{Jord}(9)=174015+1188=175203.
\]

In degree $10$ they differ in exactly fifteen entries, and the
multilinear $s$-identities have total dimension $24145$. The largest
multiplicity difference occurs for $S^{(4,3,2,1)}$: its multiplicity is
$383$ in $\SJ(10)$ and $394$ in $\mathrm{Jord}(10)$, so this isotypic
component contributes eleven copies of $S^{(4,3,2,1)}$, of total
dimension $11\cdot768=8448$. Summing the contributions from all fifteen
entries gives
\[
\dim\mathrm{Jord}(10)=1761695+24145=1785840.
\]
These differences describe the full spaces of multilinear
$s$-identities in the respective degrees; they do not, by themselves,
distinguish consequences of lower-degree $s$-identities from new
independent identities.

\appendix
\section{The nonzero characteristic factor in degree five}
\label{app:q5}

Computing the characteristic polynomials of the irreducible blocks
and applying \eqref{eq:block-characteristic}, we obtain
\begin{align*}
 q_5(t)={}&(t-8)\Bigl(t-\frac9{64}\Bigr)^4\Bigl(t-\frac{27}{32}\Bigr)^5
 \Bigl(t^2-\frac{27}{32}t+\frac{135}{1024}\Bigr)^6\\
 &\times\Bigl(t^2-\frac{207}{64}t+\frac{135}{64}\Bigr)^4
 \Bigl(t^2-\frac{117}{32}t+\frac{27}{16}\Bigr)^5\\
 &\times\Bigl(t^3-\frac{35}{16}t^2+\frac{1287}{1024}t-\frac{135}{1024}\Bigr)^5 .
\end{align*}
Its degree is $1+4+5+12+8+10+15=55=\dim\SJ(5)$, so
$\chi_5(t)=t^{65}q_5(t)$ and $\Pi_5=I-q_5(U_5)/q_5(0)$. The factors arise
from the blocks as follows, the exponents being the dimensions
$d_\lambda$: $t-8$ from $(5)$; $t^2-\tfrac{207}{64}t+\tfrac{135}{64}$ from
$(4,1)$, of rank $2$; $(t-\tfrac{27}{32})(t^2-\tfrac{117}{32}t+\tfrac{27}{16})$
from $(3,2)$, of rank $3$; $t^2-\tfrac{27}{32}t+\tfrac{135}{1024}$ from
$(3,1,1)$, of rank $2$; the cubic from $(2,2,1)$, of rank $3$; and
$t-\tfrac9{64}$ from $(2,1,1,1)$, of rank $1$. The block $(1^5)$ is zero. Direct computation of the characteristic
polynomial of the $120\times120$ regular matrix gives the same result.

\section{Exact computations and verification files}
\label{app:computations}

The computation implements Proposition~\ref{prop:row-recursion-main}
over $\Q$. Its input is a partition $\lambda$, rather than a table of
previously reported ranks. Standard tableaux are obtained by removing
the last corner box recursively, and the matrices of adjacent
transpositions use the normalization \eqref{eq:seminormal}. The
branching groups retain the positions of the last one or two entries,
including distinct groups that give the same smaller shape.

\subsection{Exact row-space calculation}

Starting from $W_\varnothing=W_{(1)}=\Q$, the program embeds bases from
degrees $n-1$ and $n-2$ into the corresponding branching groups. It
then forms the two families of rows on the right-hand side of
\eqref{eq:row-recursion}, stacks them, and computes a reduced
row-echelon basis over $\Q$. Thus every step computes exactly the
space specified by the proved recursion. Its dimension is
$r_\lambda$, and Theorem~\ref{thm:B} gives the dimension and nullity.

The implementation uses arbitrary-precision rational arithmetic
(Python \texttt{Fraction}, or \texttt{fmpq} with
\texttt{python-flint}). No floating-point rank threshold or inference
from modular rank is used. Nonzero scalar factors $1/2$ and $1/4$
may be omitted in this row-space calculation because they do not
change a span over $\Q$; they are retained when constructing the
actual Gram operators.

All $271$ blocks in degrees $1$ through $12$ were computed by this
procedure. The tables \texttt{data/blocks.csv} and
\texttt{data/dimensions.csv} record the block ranks and the resulting
dimensions. In particular, there are $56$ blocks in degree $11$
and $77$ in degree $12$, and
\[
\begin{aligned}
 \dim\SJ(11)&=19\,565\,150,& e_{11}&=20\,351\,650,\\
 \dim\SJ(12)&=236\,416\,818,& e_{12}&=242\,584\,782.
\end{aligned}
\]
Tableau counts agree with the hook-length formula, and
$\sum_{\lambda\vdash n}d_\lambda^2=n!$ in each degree.

Completed blocks are stored as rational row bases, with the
partition, field, normalization, source version and a checksum.
The saved degree summaries list every block and its checksum.
These records support reproducibility and detect accidental file
changes; the correctness of the mathematical computation rests on
the recursion and exact row reduction, rather than on checksums.

\subsection{Gram blocks and independent checks}

For completeness, we record the recursion used to construct the
Gram blocks. Let $\jmath_k\colon\Q[S_{n-k}]\to\Q[S_n]$ fix the last
$k$ positions and put
\[
 B_{n,1}=\tfrac12(I+c_n),\qquad
 B_{n,2}=\tfrac14(I+s_{n-1})(I+c_n^2).
\]
Set $U_0=I$, $T_\varnothing=I$ and $c_1=I$. Appending the last
linear or quadratic factor gives
\[
 T_{(\beta,1)}=\jmath_1(T_\beta)B_{n,1},\qquad
 T_{(\beta,2)}=\jmath_2(T_\beta)B_{n,2}.
\]
The second formula includes $(2)$, because $c_2^2=I$.
Separating the compositions by their last part gives, for $n\ge2$,
\begin{equation}
 U_n=B_{n,1}^\dagger\jmath_1(U_{n-1})B_{n,1}
     +B_{n,2}^\dagger\jmath_2(U_{n-2})B_{n,2}.
 \label{eq:Gram-recursion}
\end{equation}
For $n=1$ only the first term occurs. Every composition is retained,
including compositions that define the same operator.
Applying $\rho_\lambda$ and the branching decomposition computes a
Gram block. Adjoint matrices are formed by
\eqref{eq:block-adjoint}, or by inverting permutations, not by
ordinary transposition in a nonorthonormal seminormal basis.

In degrees at most $7$, a separate implementation constructs the
tableaux by successive addition of boxes and forms
$\sum_\alpha\rho_\lambda(T_\alpha^\dagger)\rho_\lambda(T_\alpha)$
directly. It checks the Coxeter relations and a positive invariant
form, and compares the resulting row spaces with those from
\eqref{eq:row-recursion}. The spaces agree in all $44$ blocks.
The exact characteristic polynomials are recorded in
\texttt{data/block\_polynomials.csv}; their zero-root multiplicities
are $d_\lambda-r_\lambda$.

In degrees at most $5$, direct expansion in the associative word
basis provides an additional check. The span of all commutative
binary bracketings agrees with the image of the regular Gram
matrix constructed from the $T_\alpha$. The regular characteristic
polynomials agree with the products of block polynomials in
Theorem~\ref{thm:B}, including the factorizations of $q_3$, $q_4$
and $q_5$. The degree-three projection and the degree-four examples
are also checked directly.

The comparison with the free Jordan algebra uses the published
multiplicities in \cite[Proposition~4.2]{DotsenkoHentzel}. It does
not constitute a new calculation of $\mathrm{Jord}(n)$.
The multiplicity calculations through degree $12$ and the
characteristic-polynomial calculations through degree $7$ have
different verified ranges. None of the finite-degree calculations
is used to prove Theorems~\ref{thm:A} and~\ref{thm:B}.

\subsection{Repository and reproduction}

The programs, exact tables and verification records are available
in the GitHub repository~\cite{JordanCode}. Its \texttt{README.txt}
specifies the arithmetic dependencies, commands for recomputation,
the layout of the data, and the scope of each verification.
The scripts were prepared with assistance from ChatGPT and checked
by the authors.

\clearpage
\section{Irreducible decompositions in degrees eight through twelve}
\label{app:multiplicities}

The following decompositions list the multiplicities obtained by
the exact row-space recursion. They are included here so that the
module structures in these degrees can be read without consulting
the electronic tables.

For compactness, repeated parts of a partition are expressed in
exponential notation, for example
$S^{(4,2,1^3)}=S^{(4,2,1,1,1)}$.
A module with multiplicity zero is omitted.

\textbf{Degree 8}
\begin{align*}
\mathrm{SJ}(8)\cong{}&S^{(8)} \oplus 3S^{(7,1)} \oplus 12S^{(6,2)} \oplus 9S^{(6,1^2)} \\
&{}\oplus 12S^{(5,3)} \oplus 32S^{(5,2,1)} \oplus 16S^{(5,1^3)} \oplus 10S^{(4^2)} \\
&{}\oplus 34S^{(4,3,1)} \oplus 32S^{(4,2^2)} \oplus 41S^{(4,2,1^2)} \oplus 14S^{(4,1^4)} \\
&{}\oplus 18S^{(3^2,2)} \oplus 30S^{(3^2,1^2)} \oplus 34S^{(3,2^2,1)} \oplus 26S^{(3,2,1^3)} \\
&{}\oplus 5S^{(3,1^5)} \oplus 10S^{(2^4)} \oplus 11S^{(2^3,1^2)} \oplus 6S^{(2^2,1^4)} \\
&{}\oplus S^{(2,1^6)}
\end{align*}

\textbf{Degree 9}
\begin{align*}
\mathrm{SJ}(9)\cong{}&S^{(9)} \oplus 4S^{(8,1)} \oplus 15S^{(7,2)} \oplus 12S^{(7,1^2)} \\
&{}\oplus 24S^{(6,3)} \oplus 53S^{(6,2,1)} \oplus 25S^{(6,1^3)} \oplus 22S^{(5,4)} \\
&{}\oplus 78S^{(5,3,1)} \oplus 64S^{(5,2^2)} \oplus 91S^{(5,2,1^2)} \oplus 31S^{(5,1^4)} \\
&{}\oplus 44S^{(4^2,1)} \oplus 84S^{(4,3,2)} \oplus 106S^{(4,3,1^2)} \oplus 108S^{(4,2^2,1)} \\
&{}\oplus 84S^{(4,2,1^3)} \oplus 20S^{(4,1^5)} \oplus 18S^{(3^3)} \oplus 83S^{(3^2,2,1)} \\
&{}\oplus 58S^{(3^2,1^3)} \oplus 44S^{(3,2^3)} \oplus 74S^{(3,2^2,1^2)} \oplus 38S^{(3,2,1^4)} \\
&{}\oplus 6S^{(3,1^6)} \oplus 22S^{(2^4,1)} \oplus 18S^{(2^3,1^3)} \oplus 7S^{(2^2,1^5)}\oplus S^{(2,1^7)}
\end{align*}

\newpage
\textbf{Degree 10}
\begin{align*}
\mathrm{SJ}(10)\cong{}&S^{(10)} \oplus 4S^{(9,1)} \oplus 20S^{(8,2)} \oplus 16S^{(8,1^2)} \\
&{}\oplus 35S^{(7,3)} \oplus 80S^{(7,2,1)} \oplus 40S^{(7,1^3)} \oplus 50S^{(6,4)} \\
&{}\oplus 155S^{(6,3,1)} \oplus 120S^{(6,2^2)} \oplus 168S^{(6,2,1^2)} \oplus 58S^{(6,1^4)} \\
&{}\oplus 16S^{(5^2)} \oplus 144S^{(5,4,1)} \oplus 220S^{(5,3,2)} \oplus 287S^{(5,3,1^2)} \\
&{}\oplus 260S^{(5,2^2,1)} \oplus 211S^{(5,2,1^3)} \oplus 50S^{(5,1^5)} \oplus 136S^{(4^2,2)} \\
&{}\oplus 140S^{(4^2,1^2)} \oplus 100S^{(4,3^2)} \oplus 383S^{(4,3,2,1)} \oplus 254S^{(4,3,1^3)} \\
&{}\oplus 160S^{(4,2^3)} \oplus 269S^{(4,2^2,1^2)} \oplus 151S^{(4,2,1^4)} \oplus 26S^{(4,1^6)} \\
&{}\oplus 108S^{(3^3,1)} \oplus 116S^{(3^2,2^2)} \oplus 223S^{(3^2,2,1^2)} \oplus 93S^{(3^2,1^4)} \\
&{}\oplus 143S^{(3,2^3,1)} \oplus 139S^{(3,2^2,1^3)} \oplus 53S^{(3,2,1^5)} \oplus 8S^{(3,1^7)} \\
&{}\oplus 26S^{(2^5)} \oplus 38S^{(2^4,1^2)} \oplus 26S^{(2^3,1^4)} \oplus 7S^{(2^2,1^6)} \\
&{}\oplus S^{(2,1^8)}
\end{align*}

\textbf{Degree 11}
\begin{align*}
\mathrm{SJ}(11)\cong{}&S^{(11)} \oplus 5S^{(10,1)} \oplus 24S^{(9,2)} \oplus 20S^{(9,1^2)} \\
&{}\oplus 55S^{(8,3)} \oplus 116S^{(8,2,1)} \oplus 56S^{(8,1^3)} \oplus 85S^{(7,4)} \\
&{}\oplus 270S^{(7,3,1)} \oplus 200S^{(7,2^2)} \oplus 291S^{(7,2,1^2)} \oplus 100S^{(7,1^4)} \\
&{}\oplus 66S^{(6,5)} \oplus 349S^{(6,4,1)} \oplus 495S^{(6,3,2)} \oplus 612S^{(6,3,1^2)} \\
&{}\oplus 550S^{(6,2^2,1)} \oplus 444S^{(6,2,1^3)} \oplus 112S^{(6,1^5)} \oplus 160S^{(5^2,1)} \\
&{}\oplus 500S^{(5,4,2)} \oplus 573S^{(5,4,1^2)} \oplus 320S^{(5,3^2)} \oplus 1153S^{(5,3,2,1)} \\
&{}\oplus 762S^{(5,3,1^3)} \oplus 420S^{(5,2^3)} \oplus 750S^{(5,2^2,1^2)} \oplus 423S^{(5,2,1^4)} \\
&{}\oplus 79S^{(5,1^6)} \oplus 236S^{(4^2,3)} \oplus 660S^{(4^2,2,1)} \oplus 399S^{(4^2,1^3)} \\
&{}\oplus 592S^{(4,3^2,1)} \oplus 660S^{(4,3,2^2)} \oplus 1141S^{(4,3,2,1^2)} \oplus 513S^{(4,3,1^4)} \\
&{}\oplus 578S^{(4,2^3,1)} \oplus 576S^{(4,2^2,1^3)} \oplus 239S^{(4,2,1^5)} \oplus 35S^{(4,1^7)} \\
&{}\oplus 225S^{(3^3,2)} \oplus 334S^{(3^3,1^2)} \oplus 486S^{(3^2,2^2,1)} \oplus 468S^{(3^2,2,1^3)} \\
&{}\oplus 152S^{(3^2,1^5)} \oplus 170S^{(3,2^4)} \oplus 331S^{(3,2^3,1^2)} \oplus 227S^{(3,2^2,1^4)} \\
&{}\oplus 70S^{(3,2,1^6)} \oplus 9S^{(3,1^8)} \oplus 66S^{(2^5,1)} \oplus 67S^{(2^4,1^3)} \\
&{}\oplus 34S^{(2^3,1^5)} \oplus 8S^{(2^2,1^7)} \oplus S^{(2,1^9)}
\end{align*}

\newpage
\textbf{Degree 12}
\begin{align*}
\SJ(12)\cong{}&S^{(12)} \oplus 5S^{(11,1)} \oplus 30S^{(10,2)} \oplus 25S^{(10,1^2)} \\
&{}\oplus 74S^{(9,3)} \oplus 160S^{(9,2,1)} \oplus 80S^{(9,1^3)} \oplus 145S^{(8,4)} \\
&{}\oplus 441S^{(8,3,1)} \oplus 320S^{(8,2^2)} \oplus 462S^{(8,2,1^2)} \oplus 159S^{(8,1^4)} \\
&{}\oplus 141S^{(7,5)} \oplus 704S^{(7,4,1)} \oplus 955S^{(7,3,2)} \oplus 1194S^{(7,3,1^2)} \\
&{}\oplus 1035S^{(7,2^2,1)} \oplus 844S^{(7,2,1^3)} \oplus 210S^{(7,1^5)} \oplus 76S^{(6^2)} \\
&{}\oplus 575S^{(6,5,1)} \oplus 1359S^{(6,4,2)} \oplus 1518S^{(6,4,1^2)} \oplus 810S^{(6,3^2)} \\
&{}\oplus 2814S^{(6,3,2,1)} \oplus 1833S^{(6,3,1^3)} \oplus 985S^{(6,2^3)} \oplus 1751S^{(6,2^2,1^2)} \\
&{}\oplus 1005S^{(6,2,1^4)} \oplus 196S^{(6,1^6)} \oplus 640S^{(5^2,2)} \oplus 762S^{(5^2,1^2)} \\
&{}\oplus 1056S^{(5,4,3)} \oplus 2884S^{(5,4,2,1)} \oplus 1745S^{(5,4,1^3)} \oplus 2090S^{(5,3^2,1)} \\
&{}\oplus 2205S^{(5,3,2^2)} \oplus 3837S^{(5,3,2,1^2)} \oplus 1706S^{(5,3,1^4)} \oplus 1757S^{(5,2^3,1)} \\
&{}\oplus 1793S^{(5,2^2,1^3)} \oplus 765S^{(5,2,1^5)} \oplus 123S^{(5,1^7)} \oplus 246S^{(4^3)} \\
&{}\oplus 1470S^{(4^2,3,1)} \oplus 1360S^{(4^2,2^2)} \oplus 2196S^{(4^2,2,1^2)} \oplus 951S^{(4^2,1^4)} \\
&{}\oplus 1469S^{(4,3^2,2)} \oplus 2082S^{(4,3^2,1^2)} \oplus 2877S^{(4,3,2^2,1)} \oplus 2743S^{(4,3,2,1^3)} \\
&{}\oplus 932S^{(4,3,1^5)} \oplus 765S^{(4,2^4)} \oplus 1494S^{(4,2^3,1^2)} \oplus 1081S^{(4,2^2,1^4)} \\
&{}\oplus 352S^{(4,2,1^6)} \oplus 46S^{(4,1^8)} \oplus 244S^{(3^4)} \oplus 1049S^{(3^3,2,1)} \\
&{}\oplus 795S^{(3^3,1^3)} \oplus 639S^{(3^2,2^3)} \oplus 1330S^{(3^2,2^2,1^2)} \oplus 872S^{(3^2,2,1^4)} \\
&{}\oplus 240S^{(3^2,1^6)} \oplus 572S^{(3,2^4,1)} \oplus 647S^{(3,2^3,1^3)} \oplus 340S^{(3,2^2,1^5)} \\
&{}\oplus 89S^{(3,2,1^7)} \oplus 9S^{(3,1^9)} \oplus 76S^{(2^6)} \oplus 135S^{(2^5,1^2)} \\
&{}\oplus 110S^{(2^4,1^4)} \oplus 43S^{(2^3,1^6)} \oplus 10S^{(2^2,1^8)} \oplus S^{(2,1^{10})}
\end{align*}

\section*{Data and code availability}

The source code, exact block tables, characteristic polynomials and
verification records are publicly available in the GitHub
repository~\cite{JordanCode}. Instructions for reproducing the
computations are given in its \texttt{README.txt}. These materials
are distributed through the repository.
No external dataset is required; the published free-Jordan multiplicities used for comparison are cited in the text.

\section*{Declaration of generative AI and AI-assisted technologies in the manuscript preparation process}

During this research, we used ChatGPT to  discuss proof strategies and assist in developing and debugging computational code. We also used it to improve the English, organization, and clarity of the manuscript. All mathematical arguments and computational results were checked by the authors, who reviewed and revised the AI-assisted material and take full responsibility for the content of the paper.

\section*{Declaration of competing interest}

The authors declare that they have no competing interests.

\end{document}